\documentclass[10pt,a4paper]{article}
\usepackage[latin1]{inputenc}
\usepackage{amsmath, amsfonts, amssymb, amsthm, multirow, numprint, epsfig}
\usepackage[top=1in, bottom=1in]{geometry}

\newtheorem{theorem}{Theorem}

\newtheorem{proposition}[theorem]{Proposition}
\newtheorem{corollary}[theorem]{Corollary}

\newtheorem{conjecture}{Conjecture}

\newcommand{\Res}{\text{Res}}

\newcommand{\dd}[1]{\,\textit{d}#1}

\newcommand{\pr}{\mathbb{P}}
\newcommand{\ignore}[1]{}
\npdecimalsign{.}
\nprounddigits{35}

\newcommand{\nextrow}[7]{\multirow{2}{*}{$#1$} & \multirow{2}{*}{$#2$} & \multirow{2}{*}{$#3$} & \multirow{2}{*}{$#4$} & $#5$ & \multirow{2}{*}{$#7$}\\ & & & & $#6$ & }

\title{Primes Appearing in the Tower Factorization of Integers}
\author{Patrick Devlin\\ \textit{Department of Mathematics,}\\ \textit{Rutgers University, Piscataway, NJ}\\$\mathtt{prd41@math.rutgers.edu}$ \and
Edinah K. Gnang\thanks{Work also done while this author was at Rutgers University, Piscataway, NJ}\\ \textit{School of Mathematics,}\\ \textit{Institute for Advanced Study, Princeton, NJ}\\$\mathtt{gnang@ias.edu}$}
\date{June 14, 2014}
\begin{document}
\maketitle

\begin{abstract}
In this note, we introduce and discuss a new (and still very much open!) problem in elementary number theory.  In particular, every number can be uniquely expressed as a `tower factorization' into primes, and for $q$ prime, we can let $M(q)$ denote the set of all integers whose tower factorization contains the prime $q$.  We then explore the limiting densities of these sets.  Although we are able to obtain some results, there is still much more to be done.
\end{abstract}

\section{Introduction}\label{sec:introduction}
Everybody knows---or at least readily believes---that if you pick a positive integer at random, then the probability that it's even is 1/2.  But what happens if we relax the condition that the randomly selected number must be \textit{divisible} by $2$?  For example, what if we choose to also allow numbers of the form $5^2 \cdot 7^3$, or $3 \cdot 17^{(5^2 \cdot 13^{(2 \cdot 11)})} \cdot 31$, or even $11^{(3^{(5^{2} \cdot 7)} \cdot 5^5)} \cdot 23^{19}$ to be ``divisible" by 2?  Then in general, these numbers do not contain the prime $2$ in their usual prime factorizations (perhaps not even as exponents!), but nonetheless, they all contain a $2$ somewhere in a more refined factorization.  We call this representation the \textit{tower factorization} of a number, which can be defined as follows\footnote{This representation of a number was motivated by computational considerations as discussed in \cite{Gnang}.}

\begin{center}
\fbox{\parbox[c]{5.5in}{
Let $n \geq 1$ be any arbitrary integer.  Then its \textit{tower factorization} is recursively given by
\begin{itemize}
\item[{(0)}] If $n=1$, then its tower factorization is just $1$.
\item[{(1)}] If $n > 1$, let $n=p_1 ^{e_1} \cdots p_{k} ^{e_k}$ be its usual prime factorization.  Then the tower factorization of $n$ is given by
\[
n = p_1 ^{ ( f_1 )} \cdots p_{k} ^{ ( f_k )} = \prod _{i=1} ^{k} p_i ^{(f_i)},
\]
where $f_i$ is the tower factorization of $e_i$.
\end{itemize}
}}
\end{center}
\paragraph*{}In other words, to obtain the tower factorization of $n$, we first write $n$ into its prime factorization $n = p_1 ^{e_1} \cdots p_{k} ^{e_k}$.  Then we factor each exponent $e_i > 1$ into its prime factorization $e_i = p_{i,1} ^{f_{i,1}} \cdots p_{i,k_{i}} ^{f_{i,k_{i}}}$.  We continue to factor each subsequent exponent of those prime factorizations until $n$ is writen in the form
\[
n = \prod _{i} p_{i} ^{\left( \prod_j p_{i,j} ^{\ \ (\prod \cdots )} \right)},
\]
for some prime numbers $p_{i}, p_{i,j}, \ldots$.  Then by the fundamental theorem of arithmetic, this tower factorization exists and is unique up to reordering.

\paragraph*{}With this in place, we now pose the main question of this paper:
\begin{quote}
Let $q$ be any fixed prime number (e.g., $q=2$).  Then what is the probability that a randomly selected positive integer contains the prime $q$ in its \textit{tower factorization}?
\end{quote}
To be perfectly precise, let $M (q)$ be the set of all positive integers whose tower factorization contains the prime $q$ at least once\footnote{For later notational convenience, for all $q$, we set $1 \notin M(q)$.}.  Then we wish to determine the natural density, $d(q)$, of the set $M(q)$.  That is to say, what is
\[
d(q) := \lim_{N \to \infty} \dfrac{|M(q) \cap \{1, 2, \ldots , N\}|}{N},
\]
and does such a density even exist?  (Stated more colorfully, $d(q)$ is the probability that a positive integer `chosen at random' contains the prime $q$ in its tower factorization.)

\subsubsection*{Initial thoughts}
Before continuing to discuss the problem in its generality, it is worth-while to first try to develop some intuition for the smallest possible case---namely, $q=2$.  On the one hand, the set $M(2)$ contains every even number, so right away we see that if its density $d(2)$ exists, then it would be at least $1/2$.  But as \textbf{Table \ref{table_factorizations_examples}} illustrates, $M(2)$ also contains \textit{many} odd numbers as well, and it is not immediately clear how much of a contribution these would have on the density of $M(2)$.
\begin{table}[h]
\begin{center}
\begin{tabular}{|c|c|c|}\hline
\textbf{Number, $n$} & \textbf{Tower Factorization} & \textbf{Primes $q$ with $n \in M(q)$}\\\hline
1 & 1 & $\emptyset$\\\hline
144 & $2^{(2^2)} \cdot 3^2$ & $\{2, 3\}$\\\hline
625 & $5 ^{(2^2)}$ & $\{2, 5\}$\\\hline
33787663 & $7 \cdot 13^{2 \cdot 3}$ & $\{2, 3, 7, 13\}$\\\hline
$37349 \cdot 11^{669921875} $ & $13^3 \cdot 17 \cdot 11^{(7^3 \cdot 5^{(3^2)})}$ & $\{2, 3, 5, 7, 11, 13, 17\}$\\\hline
\end{tabular}
\caption{Examples of tower factorizations and to which $M(q)$ each number belongs}
\label{table_factorizations_examples}
\end{center}
\end{table}\vspace*{-.35in}
\paragraph*{}On the other hand, it is easy to see that $M(2)$ does not contain any odd square-free integers.  Therefore, since the odd square-free integers have density $(2/3) \cdot (6 / \pi ^2) = 4 / \pi ^2$, we have that if $d(2)$ exists, then it is bounded by $d(2) \leq 1 - 4 / \pi ^2 = 0.592715\ldots$.  Thus, if the limiting density of $M(2)$ exists, we intuitively find it would satisfy $0.5 \leq d(2) \leq 0.592716$.
\begin{figure}[htb]
\centering
\includegraphics[width=\textwidth]{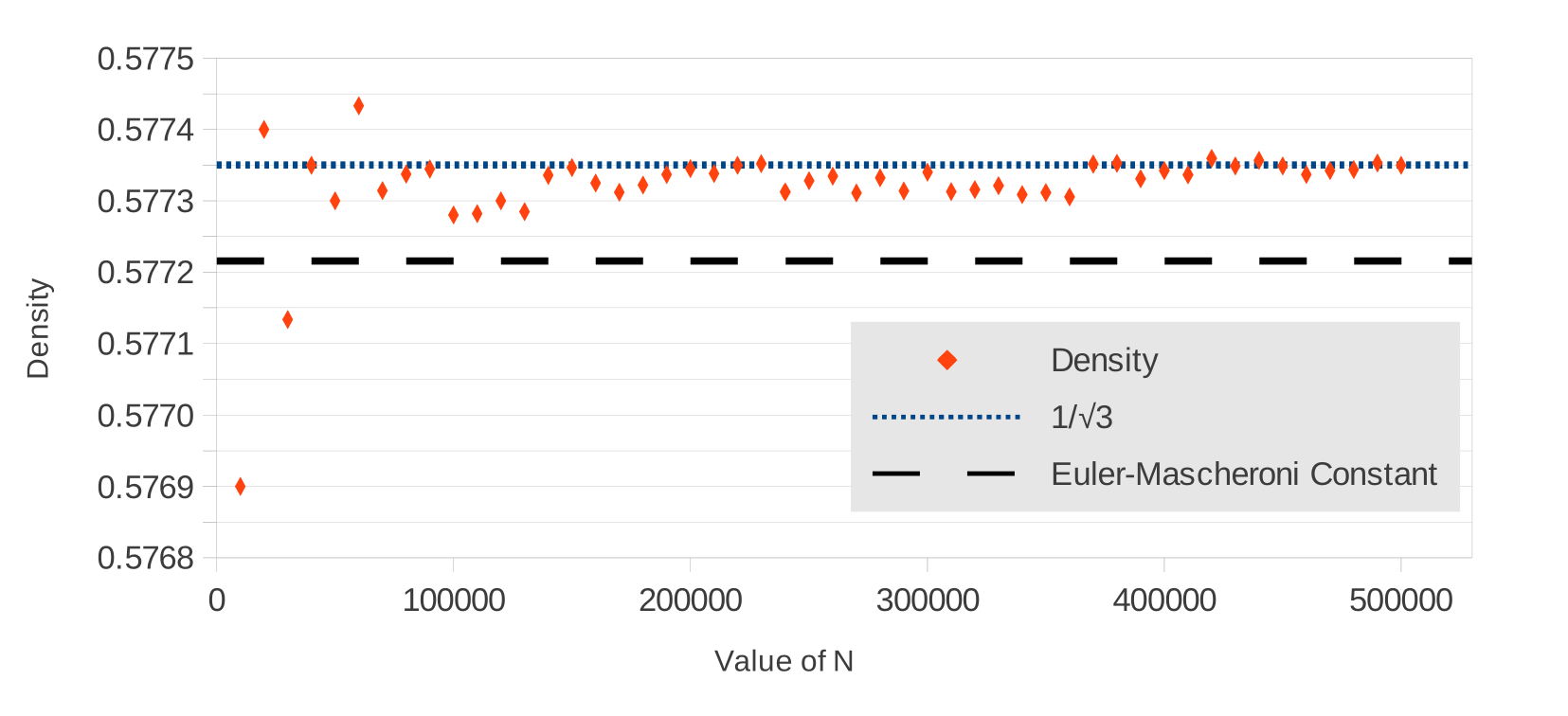}
\caption{Density of $M(2)$ in $\{1, 2, \ldots , N\}$ for various $N$}
\label{fig:experimental_data}
\end{figure}
\paragraph*{}Beyond this first analysis, the intuitive arguments become more difficult.  For instance, does the density $d(2)$ exist, and if so is it closer to $0.5$ or to $0.592716$?  Although-- or perhaps \textit{because}-- these sort of questions are difficult to answer by `blind' intuition, looking at computer data can be particularly insightful.  Slapping together some na\"ive code, we readily compute the density $|M(2) \cap \{1, 2, \ldots , N\}|/N$ for various values $N$ as plotted in \textbf{Figure \ref{fig:experimental_data}}.
\paragraph*{}Now that we can `see' how the density of $M(2)$ in $\{1, 2, \ldots, N\}$ changes as $N$ grows larger and larger, we are able to come up with better guesses.  For example, although at this point, we have virtually no rigorous results whatsoever, the experimental data \textit{suggests} that the limit for $d(2)$ converges and that the first digits of $d(2)$ \textit{seem to be} about $0.5773 \ldots$, which leads to several tempting conjectures.  Before continuing, we invite and encourage the reader to spend a moment or so to form an opinion on what they expect about this limiting density.

\subsection*{Outline of paper}
It is first best to note that the arguments in this paper are all very elementary, and the authors are certainly not number theorists.  Therefore, those bored with the presentation can take solace in that they may be able to surpass our results easily (which they are very encouraged to do!).
\paragraph*{}We begin in \textbf{Section \ref{sec:formulas}} by deriving a few formulas for the densities $d(q)$ [establishing their existence] with the result of \textbf{Proposition \ref{tempting_representation}} being particularly inviting.  Then in \textbf{Section \ref{sec:bounds}}, we use these formulas to get a few bounds on $d(q)$ and establish their asymptotics.  We also settle some tempting conjectures about the value of $d(2)$.  We conclude with a few obvious open questions, and we provide some numerically computed bounds for $d(q)$ in an appendix.

\section{Representations for $d(q)$}\label{sec:formulas}
As the name of this section might suggest, we will now show that for all primes $q$, the limiting density $d(q)$ is given by any one of several convergent formulas, and therefore it does in fact exist.  However, as the existence of subsequent sections might suggest, these formulas are not in a ``closed enough" form to just call it a day after their derivation!
\paragraph*{}The formulas for these limiting densities are found by what has now become a routine `probabilistic' approach, and so for ease of reading, proofs are confined to their most important aspects.  The argument has one `hand-wavy' part (which we will be kind enough to point out), but we believe that this presentation is the most intuitive one, and we trust our readers will forgive us.

\begin{theorem}\label{formulas}
Let $q$ be any arbitrary prime, and let $M(q)$ be as previoulsy defined.  Then the density $d(q)$ exists and is given by
\[
d(q) = 1 - \left( 1- \dfrac{1}{q}\right) \cdot \prod_{\text{$p \neq q$, prime}} \left[1 - \left( 1-\dfrac{1}{p} \right )\sum_{m \in M(q)} \dfrac{1}{p^m} \right].
\]
Moreover, if $M^{c} (q)$ is defined as\footnote{Note that $\{0, 1\} \subseteq M^{c} (q)$} $M^{c} (q) := \{z \in \{0, 1, \ldots \} : z \notin M(q)\}$, then we also have
\begin{eqnarray*}
d(q) &=& 1 - \left( 1- \dfrac{1}{q}\right) \cdot \prod_{\text{$p \neq q$, prime}} \left[\left( 1-\dfrac{1}{p} \right )\sum_{m \in M^c (q)} \dfrac{1}{p^m} \right]\\
&=& 1 - \left(\sum_{m \in M^c (q)} q^{-m}\right)^{-1} \cdot \prod_{\text{$p$, prime}} \left[\left( 1-\dfrac{1}{p} \right )\sum_{m \in M^c (q)} \dfrac{1}{p^m} \right].
\end{eqnarray*}
\end{theorem}
\begin{proof}
Let $n \in \{1, \ldots , N\}$ be chosen uniformly at random (we will let $N$ go to infinity).  Then for all primes $p$ and all nonnegative integers $m$, let $D(p,m)$ be the event that $p ^m$ exactly divides $n$ (that is $p^m$ divides $n$, but $p^{m+1}$ does not).  Then the probability that $n$ is in $M(q)$ is given by:
\begin{eqnarray*}
\pr (n \in M(q)) &=& \pr \left [ (\text{$q$ divides $n$}) \cup \left( \text{$p^m$ exactly divides $n$, where $p \neq q$, and $m \in M(q)$} \right )  \right ]\\
&=& \pr \left [ (\text{$q$ divides $n$}) \cup \left( \bigcup_{\text{$p \neq q$, prime \ }} \bigcup_{m \in M(q)} D(p,m) \right )  \right ].
\end{eqnarray*}
These events are not disjoint, which makes this probability difficult to deal with.  However, by considering $\pr (n \notin M(q)) = 1 - \pr(n \in M(q))$, we have the simplification
\begin{eqnarray*}
1 - \pr (n \in M(q)) &=& \pr \left ( \left [ (\text{$q$ divides $n$}) \cup \left( \bigcup_{\text{$p \neq q$, prime \ }} \bigcup_{m \in M(q)} D(p,m) \right ) \right] ^c  \right )\\
&=& \pr \left ( (\text{$q$ does not divide $n$}) \cap \left[ \bigcap_{\text{$p \neq q$, prime \ }} \left( \bigcup_{m \in M(q)} D(p,m) \right) ^c \right ]  \right ).
\end{eqnarray*}
As $N$ goes to infinity, we may treat these events as independent\footnote{This is the aforementioned `hand-wavy' part, but it is neither particularly difficult nor enlightening to make this step rigorous.  In fact, this would be a good exercise for any undergraduate students reading this.  Note that in going to the infinite product, attention needs to be given to show that the error terms do not accumulate.}.  Thus, as $N \to \infty$ we have
\begin{eqnarray*}
1- d(q) &=& \lim_{N \to \infty} \pr (\text{$q$ does not divide $n$}) \cdot \pr \left[ \bigcap_{\text{$p \neq q$, prime \ }} \left( \bigcup_{m \in M(q)} D(p,m) \right) ^c \right ]\\
&=& \lim_{N \to \infty}  \pr (\text{$q$ does not divide $n$}) \cdot \prod_{\text{$p \neq q$, prime}} \left[1 - \pr \left( \bigcup_{m \in M(q)} D(p,m) \right) \right].
\end{eqnarray*}
For any fixed $p$, the events $D(p,m)$ are clearly disjoint.  Finally, as $N$ goes to infinity, we have
\begin{eqnarray*}
\lim_{N \to \infty} \pr (D (p,m)) &=& \lim_{N \to \infty} \pr (\text{$p^m$ divides $n$}) \cdot \pr (\text{$p^{m+1}$ does not divide $n$ given $p^m$ does})\\
&=& \left ( \dfrac{1}{p^m} \right) \cdot \left( 1 - \dfrac{1}{p}\right).
\end{eqnarray*}
Thus, our formula simplifies to
\begin{eqnarray*}
d(q) &=& 1 - \lim_{N \to \infty} \pr (\text{$q$ does not divide $n$}) \cdot \prod_{\text{$p \neq q$, prime}} \left[1 - \sum_{m \in M(q)} \left ( \dfrac{1}{p^m} \right) \cdot \left( 1 - \dfrac{1}{p}\right) \right]\\
&=& 1 - \left( 1- \dfrac{1}{q}\right) \cdot \prod_{\text{$p \neq q$, prime}} \left[1 - \left( 1-\dfrac{1}{p} \right )\sum_{m \in M(q)} \dfrac{1}{p^m} \right].
\end{eqnarray*}
\paragraph*{}On the other hand, since we have
\[
1 - \pr \left(\bigcup_{0 \leq m \in M(q)} D(p,m) \right) = \pr \left( \bigcup _{0 \leq m \notin M(q)} D(p,m) \right),
\]
then if we define $M^c (q) = \{z \in \{0, 1, \ldots \} : z \notin M(q)\}$ (with $\{0,1\} \subseteq M^c(q)$), we have
\begin{eqnarray*}
d(q) &=& 1- \lim_{N \to \infty}  \pr (\text{$q$ does not divide $n$}) \cdot \prod_{\text{$p \neq q$, prime}} \left[\pr \left( \bigcup _{0 \leq m \notin M(q)} D(p,m) \right) \right] \\
&=& 1 - \lim_{N \to \infty}  \pr (\text{$q$ does not divide $n$}) \cdot \prod_{\text{$p \neq q$, prime}} \left[\sum_{m \in M^c (q)} \left ( \dfrac{1}{p^m} \right) \cdot \left( 1 - \dfrac{1}{p}\right) \right]\\
&=& 1 - \left( 1- \dfrac{1}{q}\right) \cdot \prod_{\text{$p \neq q$, prime}} \left[\left( 1-\dfrac{1}{p} \right )\sum_{m \in M^c (q)} \dfrac{1}{p^m} \right],
\end{eqnarray*}
as desired.
\end{proof}
Having obtained this, we are then able to get very concise and beautiful representations for these densities such as the following\footnote{In fact, those unhappy with our previous proof could directly prove this result instead and use it to derive our previous formulas.}.
\begin{proposition}\label{tempting_representation}
If $I(n)$ is the indicator function for the event $n \in M^c (q)$, then we have
\[
1-d(q) = \prod_{\text{$p$ prime}} \left[ \left(1 - \dfrac{1}{p}\right) \sum_{m=0} ^{\infty} \dfrac{I(p^m)}{p^m} \right].
\]
\end{proposition}
\begin{proof}
Note that $I(n)$ can be defined as the multiplicative function such that for all $n \geq 1$,
\[
I(p^n) = I(n) \cdot I(p),
\]
where $I(p)$ coincides with the indicator function for $p \neq q$.  Thus the last theorem can be rewritten
\begin{eqnarray*}
1-d(q) &=& \left( 1- \dfrac{1}{q}\right) \cdot \prod_{\text{$p \neq q$, prime}} \left[\left( 1-\dfrac{1}{p} \right )\sum_{m \in M^c (q)} \dfrac{1}{p^m} \right]\\
&=& \left( 1- \dfrac{1}{q}\right) \cdot \prod_{\text{$p \neq q$, prime}} \left[\left( 1-\dfrac{1}{p} \right )\sum_{m=0}^{\infty} \dfrac{I(p^m)}{p^m} \right] = \prod_{\text{$p$ prime}} \left[\left( 1-\dfrac{1}{p} \right )\sum_{m=0}^{\infty} \dfrac{I(p^m)}{p^m} \right],
\end{eqnarray*}
as desired.
\end{proof}

Because $I(m)$ as defined is a multiplicative function that behaves so simply on the primes, this above representation will likely lead the analytically inclined reader to consider the function $f(s) := \sum_{m\geq 1} I(m) m^{-s}$ [and perhaps $f(s)/\zeta(s)$].  Many promising things can be said and done with this function, and the authors believe that this approach \textit{should} be quite revealing.  However, as the authors are neither number theorists nor analysts, we personally were unable to exploit this formula for anything truly useful (though not for want of effort).

\ignore{
\begin{conjecture}
If $I(n)$ is the indicator function for the event $n \in M^c (q)$, then we conjecture\footnote{This is a `conjecture' because we never bothered to prove the details of this argument (because the authors were unable to get anything useful from this representation).}
\begin{eqnarray*}
1-d(q) &=& \prod_{\text{$p$ prime}} \left[ \lim_{s \to 1^+} \left(1 - \dfrac{1}{p^s}\right) \sum_{m=0} ^{\infty} \dfrac{I(p^m)}{p^{ms}} \right] = \lim_{s \to 1^{+}} \prod_{\text{$p$ prime}} \left[ \left(1 - \dfrac{1}{p^s}\right) \sum_{m=0} ^{\infty} \dfrac{I(p^m)}{p^{ms}} \right]\\
&=& \lim_{s \to 1^{+}} \dfrac{1}{\zeta(s)} \sum_{m=1} ^{\infty} \dfrac{I(m)}{m^{s}} = \lim_{s \to 1^{+}} (1-s) \sum_{m=1} ^{\infty} \dfrac{I(m)}{m^{s}} = \Res_{s=1} f(s),
\end{eqnarray*}
where $f(s)$ is the analytic continuation of $\sum_{m=1} ^{\infty} I(m)m^{-s}$.
\end{conjecture}
}

\section{Bounds and asymptotics}\label{sec:bounds}
Armed with the formulas of \textbf{Theorem \ref{formulas}}, we dive into some bounds, which ultimately lead us to a very good asymptotic understanding of $d(q)$.  The first of these results follows so readily from \textbf{Theorem \ref{formulas}} that it needs no proof.
\begin{theorem}\label{bounds}
For each prime $p \neq q$, let $S(p)$, $T(p)$, $A(p)$, and $B(p)$ be integer subsets such that $S(p) \subseteq M(q) \subseteq T(p)$ and $A(p) \subseteq M^{c} (q) \subseteq B(p)$.  Then we have
\begin{eqnarray*}
d(q) &\geq& 1 - \left( 1- \dfrac{1}{q}\right) \cdot \prod_{\text{$p \neq q$, prime}} \left[1 - \left( 1-\dfrac{1}{p} \right )\sum_{s \in S(p)} \dfrac{1}{p^s} \right]\\
d(q) &\leq& 1 - \left( 1- \dfrac{1}{q}\right) \cdot \prod_{\text{$p \neq q$, prime}} \left[1 - \left( 1-\dfrac{1}{p} \right )\sum_{t \in T(p)} \dfrac{1}{p^t} \right],\\
d(q) &\leq& 1 - \left( 1- \dfrac{1}{q}\right) \cdot \prod_{\text{$p \neq q$, prime}} \left[\left( 1-\dfrac{1}{p} \right )\sum_{a \in A(p)} \dfrac{1}{p^a} \right],\\
d(q) &\geq& 1 - \left( 1- \dfrac{1}{q}\right) \cdot \prod_{\text{$p \neq q$, prime}} \left[\left( 1-\dfrac{1}{p} \right )\sum_{b \in B(p)} \dfrac{1}{p^b} \right].
\end{eqnarray*}
\end{theorem}

After a little thought, we are able to use this result to obtain the following easier bounds.  These are more useful because they involve only finite sums and products, and yet in the limit they still squeeze together.
\begin{proposition}\label{easy_bound}
Let $P$ be any set of primes with $q \notin P$, let $S \subseteq M(q)$, let $A \subseteq M^c (q) \subseteq B$.  Then we have
\begin{eqnarray*}
d(q) &\geq& 1 - \left( 1- \dfrac{1}{q}\right) \cdot \prod_{p \in P} \left[1 - \left( 1-\dfrac{1}{p} \right )\sum_{s \in S} \dfrac{1}{p^s} \right],\\
d(q) &\geq& 1 - \left( 1- \dfrac{1}{q}\right) \cdot \prod_{p \in P} \left[\left( 1-\dfrac{1}{p} \right )\sum_{b \in B} \dfrac{1}{p^b} \right],\\
d(q) &\leq& 1 - \dfrac{q^{q}-q^{q-1}}{q^{q}-1} \cdot \dfrac{1}{\zeta(q)} \prod_{p \in P} \left[ \dfrac{\left( 1 - \dfrac{1}{p}\right)}{1 - \dfrac{1}{p^q}} \sum_{a \in A}\dfrac{1}{p^a} \right].
\end{eqnarray*}
Moreover, if $q \in S$, then we also have
\[
d(q) \geq 1 - \dfrac{q^{q+1}-q^q}{q^{q+1} - 1} \cdot \dfrac{1}{\zeta (q+1)} \prod_{p \in P} \left[ \dfrac{1 - \left(1-\dfrac{1}{p} \right ) \displaystyle \sum_{s \in S} \dfrac{1}{p^s}}{ 1 - \dfrac{1}{p^{q+1}}} \right].
\]
\end{proposition}
\begin{proof}
The first two inequalities follow immediately from \textbf{Theorem \ref{bounds}} by using the set families
\[
S(p) := \begin{cases} S, \qquad &\text{for $p \in P$}\\ \emptyset, \qquad &\text{for $p \notin P$}\end{cases}, \qquad \text{and} \qquad B(p) := \begin{cases} B, \qquad &\text{for $p \in P$}\\ \{0, 1, 2, \ldots\}, \qquad &\text{for $p \notin P$}\end{cases}.
\]
The third inequality follows again from \textbf{Theorem \ref{bounds}} by using the set family
\[
A(p) := \begin{cases} A, \qquad &\text{for $p \in P$}\\ \{0, 1, 2, \ldots , q-1\}, \qquad &\text{for $p \notin P$}\end{cases}.
\]
To prove the final inequality, note from \textbf{Theorem \ref{formulas}} we have
\begin{eqnarray*}
d(q) &=& 1 - \dfrac{1- \dfrac{1}{q}}{1 - \dfrac{1}{q^{q+1}}} \cdot \prod_{\text{$p \neq q$, prime}} \dfrac{\left[1 - \left(1-\dfrac{1}{p} \right )\sum_{m \in M (q)} \dfrac{1}{p^m} \right]}{ 1 - \dfrac{1}{p^{q+1}}} \prod_{p} \left[ 1 - \dfrac{1}{p^{q+1}} \right]\\
&=& 1 - \dfrac{q^{q+1}-q^q}{q^{q+1} - 1} \cdot \dfrac{1}{\zeta (q+1)} \prod_{\text{$p \neq q$, prime}} \dfrac{\left[1 - \left(1-\dfrac{1}{p} \right )\sum_{m \in M (q)} \dfrac{1}{p^m} \right]}{ 1 - \dfrac{1}{p^{q+1}}}.
\end{eqnarray*}
Now for all primes $p \neq q$, since $q \in S$ and $p \geq 2$, we have
\[
\dfrac{\displaystyle  1 - \left(1 - \dfrac{1}{p} \right) \sum_{m \in M} \dfrac{1}{p^m}}{1-\dfrac{1}{p^{q+1}}} \leq \dfrac{\displaystyle  1 - \left(1 - \dfrac{1}{p} \right) \sum_{s \in S} \dfrac{1}{p^s}}{1-\dfrac{1}{p^{q+1}}} \leq \dfrac{\displaystyle 1 - \left(1 - \dfrac{1}{p} \right) \dfrac{1}{p^q}}{1-\dfrac{1}{p^{q+1}}} \leq \dfrac{\displaystyle 1 - \dfrac{1}{p} \cdot \dfrac{1}{p^q}}{1-\dfrac{1}{p^{q+1}}} = 1,
\]
which we are then able to use to truncate the infinite product as desired.
\end{proof}
For fixed values of $q$, these allow us to use the computer to rigorously calculate digits of $d(q)$, and they also allow the following bounds.
\begin{proposition}\label{asymptotic_bounds}
For all $q$, we have
\[
1 - \dfrac{q^{q+1}-q^{q}}{q^{q+1}-1} \cdot \dfrac{1}{\zeta (q+1)} \leq d(q) \leq 1 - \dfrac{q^{q}-q^{q-1}}{q^{q}-1} \cdot \dfrac{1}{\zeta(q)}.
\]
Therefore, we have
\[
\dfrac{1}{2^{q+1}} (1 - q^{-1} ) - \dfrac{1}{q^q} \leq d(q) - \dfrac{1}{q} \leq \dfrac{1}{2^{q}} (1 + q^{-1}).
\]
\end{proposition}
\begin{proof}
The first inequalities follow immediately from the last proposition by taking $P = \emptyset$.  The second inequalities then follow from routine computations after using the elementary bounds
\[
1 + \dfrac{1}{2^{s}} \leq \zeta(s) \leq 1 + \dfrac{1}{2^{s}} + \int_{2}^{\infty} \dfrac{\dd{x}}{x^s} = 1 + \dfrac{2+s-1}{2^s (s-1)},
\]
which are valid for all real values of $s$ greater than $1$.
\end{proof}
This shows that for $q$ large, the value of $d(q)$ is very close to $1/q$.  That is, the probability that a number contains $q$ in its tower factorization is very close to the probability that it is divisible by $q$.  But also note that the value is bounded away from $1/q$ by an additive term on the order of $2^{-q}$.  This can be explained by noting that $2^{-q}$ is essentially the probability that the term $2^q$ appears in the prime factorization of $n$, and since $2$ is the smallest prime, it makes sense that the contribution due to terms of this type be fundamentally larger than the rest.
\paragraph*{}These bounds are sufficiently tight to give another very believable result:
\begin{corollary}
The sequence $d(q)$ is strictly decreasing, and it decreases to $0$.
\end{corollary}

\subsection*{Numerics for $d(2)$}
Let us now return to address the original question: ``what is the value of $d(2)$?"  Using \textbf{Proposition \ref{asymptotic_bounds}} provides the bounds $0.5246243585 \leq d(2) \leq 0.5947152656$, which is not yet refined enough to rule out tempting conjectures like $d(2) = \gamma$ or $d(2) = 1/\sqrt{3}$ that our initial data from \textbf{Section \ref{sec:introduction}} may have suggested.

\paragraph*{}Nonetheless, we can use the bounds of \textbf{Proposition \ref{easy_bound}} to write a program that (eventually) calculates $d(2)$ to within arbitrarily precision.  More specifically, using these bounds with $A = M^c (2) \cap \{1, 2, \ldots , 20\}$, with $S = M(2) \cap \{1, 2, \ldots , 20\}$, and with $P$ being the set of the first $25,000$ primes yields
\[
0.577350376 < d(2) < 0.577350486,
\]
and since $1/\sqrt{3} = 0.57735026\ldots$ and $\gamma = 0.577215 \ldots$, this definitively (and perhaps anticlimactically) shows that $\gamma < 1/\sqrt{3} < d(2)$, which disproves any such conjecture.  It is curious to note though how very close $d(2)$ is to $1/\sqrt{3}$, and the authors have no explanation for this.

\paragraph*{}Using this same technique, we are able to compute numeric bounds on various other values of $d(q)$, which we present in the appendix.

\section{Conclusion}
So what is the value of $d(2)$?  Apparently it's just slightly larger than $1/\sqrt{3}$, but what is an exact answer?  Is $d(2)$ algebraic?  Is it expressible in terms of elementary functions or more satisfying limits, or is it possible that perhaps $d(2)$ is in some sense its own transcendental mathematical constant?   Unfortunately, after many attempts, the authors were unable to make headway on any of these questions let alone the corresponding questions for $d(q)$ in general.
\paragraph*{}Nonetheless, the authors believe this problem is very interesting---especially because the representation of $d(q)$ in \textbf{Proposition \ref{tempting_representation}} is so tempting.  The problem has a certain fractal-like self-similarity, and it feels like some beautiful idea is just waiting to be applied.  The authors hope for progress on the problem, and we wish our readers the best with these loose ends.

\bibliography{mybib}

\begin{thebibliography}{1}

\bibitem{Gnang}
Patrick Devlin and Edinah~K. Gnang.
\newblock Some integer formula encodings and related algorithms.
\newblock {\em Advances in Applied Mathematics}, 51:536--541, September 2013.

\end{thebibliography}

\newpage
\section{Appendix}
Here we tabulate numeric bounds found on $d(q)$ for various values of $q$.  These were found by using \textbf{Proposition \ref{easy_bound}} and a simple Maple script.  The floating point values in the fifth column are a rigorous lower bound for $d(q)$ (on top) and a rigorous upper bound for $d(q)$ (on bottom).  More complete values for these bounds have been computed, but they are truncated to just 35 digits here.

\paragraph*{}The columns $p$, $a$, and $s$ are parameters for the algorithms used.  These parameters correspond to the size of $P$, (roughly) the size of $A$, and (roughly) the size of $S$ as in \textbf{Proposition \ref{easy_bound}}.  Notably, $p$ is the number of primes used in the estimation, which seems to matter much more than the size of $A$ [affecting the upper bound] or $S$ [affecting the lower bound].  Data and code are available on request.
\begin{center}
\begin{tabular}{|c|c|c|c|n{1}{36}|c|}\hline
$q$ & $p$ & $a$ & $s$ & \text{Bounds (First 35 Digits [more available])} & Digits Known\\\hline\hline
\nextrow{2}{25000}{20}{20}
{0.577350376056807813001171222749099027793826886470544627211675882194082714}
{0.577350485047678584952747233500637548585202776756754996491063963297074978}
{6}\\\hline
\nextrow{3}{6000}{100}{100}{0.388807379263994405608}{0.388807379271511226974}{10}\\\hline
\nextrow{5}{5000}{100}{100}{0.2151189846955856203278886157360448757908}{0.2151189846955856203310804143009889413781}{19}\\\hline
\nextrow{7}{2500}{100}{100}
{0.1465008912284380428191169151038108078952016850611}{0.1465008912284380428191169151051370285686287787482}{29}\\\hline
\nextrow{11}{2000}{200}{200}
{0.09113458105567412165027231631480880531869134253505}
{0.09113458105567412165027231631480880531869134616405}
{44}\\\hline
\nextrow{13}{2000}{200}{200}
{0.0769798105202947775196592008915016896290643581467495222992}
{0.0769798105202947775196592008915016896290643581467495324792}
{52}\\\hline
\nextrow{17}{2000}{200}{200}
{0.0588271246021194036767367088849109584242431438294286088658355576978825563}
{0.0588271246021194036767367088849109584242431438294286088658358505456589114}
{60}\\\hline
\nextrow{19}{2000}{200}{200}
{0.0526324829734675179643555340250633283774469991562117016273662733238280360295}
{0.0526324829734675179643555340250633283774469991562117016273665680982357656219}
{60}\\\hline
\nextrow{23}{2000}{200}{200}
{0.0434783178894840833442936676959388965942544846621537400365341488199880437161}
{0.0434783178894840833442936676959388965942544846621537400365344464424578523589}
{60}\\\hline
\nextrow{29}{2000}{200}{200}
{0.0344827595199070388388844049792239641953873954627365728107837569520705630160}
{0.0344827595199070388388844049792239641953873954627365728107840573734977701028}
{60}\\\hline
\nextrow{31}{2000}{200}{200}
{0.0322580647414500545950163257009657078579131023779285736837851195575316015814}
{0.0322580647414500545950163257009657078579131023779285736837854206711740787331}
{60}\\\hline
\nextrow{37}{2000}{200}{200}
{0.0270270270305666835231340923503436156463771224864202885139864189587084207034}
{0.0270270270305666835231340923503436156463771224864202885139867216999922085428}
{60}\\\hline
\nextrow{41}{2000}{200}{200}
{0.0526324829734675179643555340250633283774469991562117016273662733238280360295}
{0.0526324829734675179643555340250633283774469991562117016273665680982357656219}
{60}\\\hline
\nextrow{43}{2000}{200}{200}
{0.0243902439026608523841187301974189252492456339087527260100490728868251728131}
{0.0243902439026608523841187301974189252492456339087527260100493764485460440884}
{60}\\\hline
\nextrow{47}{2000}{200}{200}
{0.0212765957446843281878803870513876380451765585785993642707367688487106754059}
{0.0212765957446843281878803870513876380451765585785993642707370733792455494618}
{60}\\\hline
\end{tabular}
\end{center}

\begin{center}
\begin{tabular}{|c|c|c|c|n{1}{36}|c|}\hline
$q$ & $p$ & $a$ & $s$ & \text{Bounds (First 35 Digits [More available])} & Digits Known\\\hline\hline
\nextrow{53}{2000}{200}{200}
{0.0188679245283019412562238832810092145871135099055870468782572222350689036016}
{0.0188679245283019412562238832810092145871135099055870468782575275150636666011}
{60}\\\hline
\nextrow{59}{2000}{200}{200}
{0.0169491525423728822085928950180309408613925856773051105788352397597828390175}
{0.0169491525423728822085928950180309408613925856773051105788355456368049711755}
{60}\\\hline
\nextrow{61}{2000}{200}{200}
{0.0163934426229508198854168208219767903619261302180219681216219806030511287992}
{0.0163934426229508198854168208219767903619261302180219681216222866529828268882}
{60}\\\hline
\nextrow{67}{2000}{200}{200}
{0.0149253731343283582122927865384419386637802338726817463136753178858896539878}
{0.0149253731343283582122927865384419386637802338726817463136756243926122949098}
{60}\\\hline
\nextrow{71}{2000}{200}{200}
{0.0140845070422535211269693391067176967649594829351580201486906371592417410550}
{0.0140845070422535211269693391067176967649594829351580201486909439276005088579}
{60}\\\hline
\nextrow{73}{2000}{200}{200}
{0.0136986301369863013699152280583923981060340353762078282187845491706626305503}
{0.0136986301369863013699152280583923981060340353762078282187848560590872921958}
{60}\\\hline
\nextrow{79}{2000}{200}{200}
{0.0126582278481012658227856268365541662938488265742992373710486955727305542222}
{0.0126582278481012658227856268365541662938488265742992373710490027848771828103}
{60}\\\hline
\nextrow{83}{2000}{200}{200}
{0.0120481927710843373493976414373571004130184548993846798041887017697123434429}
{0.0120481927710843373493976414373571004130184548993846798041890091716718522464}
{60}\\\hline
\nextrow{89}{2000}{200}{200}
{0.0112359550561797752808988772032116167559082693745180629757011563061883886497}
{0.0112359550561797752808988772032116167559082693745180629757014639608762829085}
{60}\\\hline
\nextrow{97}{2000}{200}{200}
{0.0103092783505154639175257731989992019340708096559895437616545481014641811975}
{0.0103092783505154639175257731989992019340708096559895437616548560444882403055}
{60}\\\hline
\nextrow{101}{2000}{200}{200}
{0.0099009900990099009900990099011853616102032207443407084385975285689705070493}
{0.0099009900990099009900990099011853616102032207443407084385978366390337675019}
{60}\\\hline
\end{tabular}
\end{center}
\end{document}